\documentclass{amsart}

\pdfoutput=1

\usepackage{microtype}
\usepackage{latexsym}
\usepackage{amssymb}
\usepackage{amsmath}
\usepackage{amsthm}
\usepackage{amsfonts}
\usepackage{color}
\usepackage{enumitem}
\usepackage{mathtools}
\usepackage{xfrac}
\usepackage[utf8]{inputenc}

\AtBeginDocument{%
   \def\MR#1{}
}

\usepackage[dvipsnames,svgnames,x11names]{xcolor}
\definecolor{darkred}{RGB}{139,0,0}
\definecolor{darkblue}{RGB}{0,0,139}
\definecolor{darkgreen}{RGB}{0,100,0}
\usepackage{hyperref}
\hypersetup{
	colorlinks   = true, 
	urlcolor     = darkred, 
	linkcolor    = darkred, 
	citecolor   = darkgreen}

\setenumerate[1]{label=(\roman*),}
\setlist[enumerate]{leftmargin=*}
\setlist[itemize]{leftmargin=*}

\usepackage{pictexwd,dcpic}

\usepackage{graphicx}
\usepackage{subfigure}

\usepackage{psfrag}
\usepackage{hyperref}
\usepackage{comment}
\usepackage[capitalise,noabbrev]{cleveref}

\usepackage{tikz}
\usetikzlibrary{matrix,calc,positioning,arrows,decorations.pathreplacing,decorations.pathmorphing,patterns,arrows}
\usepackage{tikz-cd}

\theoremstyle{theorem}
\newtheorem{theorem}{Theorem}[section]
\newtheorem{lemma}[theorem]{Lemma}
\newtheorem{conjecture}[theorem]{Conjecture}
\newtheorem{proposition}[theorem]{Proposition}
\newtheorem{corollary}[theorem]{Corollary}
\theoremstyle{definition}

\newtheorem{remark}[theorem]{Remark}

\numberwithin{equation}{section}

\newcommand{\RR}{\mathbb{R}}

\newcommand{\DD}{\mathbb{D}}
\newcommand{\ZZ}{\mathbb{Z}}
\newcommand{\sphere}{\mathbb{S}}

\newcommand{\subm}{P}
\DeclareMathOperator{\Iso}{Isom}
\DeclareMathOperator{\Diff}{Diff}
\DeclareMathOperator{\Homeo}{Homeo}

\newcommand{\oO}{\mathrm{O}}
\newcommand{\oH}{\mathrm{H}}
\newcommand{\id}{\mathrm{id}}

\newcommand{\ra}{\rightarrow}

\newcommand{\fol}{\mathcal{F}}
\calclayout

\makeatletter
\patchcmd{\@setaddresses}{\indent}{\noindent}{}{}
\patchcmd{\@setaddresses}{\indent}{\noindent}{}{}
\patchcmd{\@setaddresses}{\indent}{\noindent}{}{}
\patchcmd{\@setaddresses}{\indent}{\noindent}{}{}
\makeatother

\begin{document}



\title{Isoparametric foliations and bounded geometry}


\author{Manuel Krannich}
\address{Department of Mathematics, Karlsruhe Institute of Technology, Karlsruhe, Germany}
\email{krannich@kit.edu}
\author{Alexander Lytchak}
\address{Department of Mathematics, Karlsruhe Institute of Technology, Karlsruhe, Germany}
\email{alexander.lytchak@kit.edu}
\author{Marco Radeschi}
\address{Dipartimento di matematica ``G. Peano'', Università degli Studi di Torino, Torino, Italy}
\email{marco.radeschi@unito.it}

\begin{abstract}
We prove that there are only finitely many isoparametrically foliated closed connected Riemannian manifolds with bounded geometry, fixed dimension $n\neq5$, and finite fundamental group, up to foliated diffeomorphism. In addition, we construct various infinite families of isoparametric foliations that are mutually not foliated diffeomorphic, for instance on a fixed sphere. \end{abstract}

\maketitle



\section{Introduction}
An \emph{isoparametric function} on a complete Riemannian manifold $(M,g)$ is a nonconstant smooth function $f:M\to \RR$ that satisfies 
\begin{equation} \label{eq: 1}
\|\nabla f\|^2=b(f) \quad\text{and}\quad  \Delta f=a(f)\,.
\end{equation}
for a smooth function $b: f(M)\to \RR$ and a continuous function $a:f(M)\to \RR$. Geometrically, the first equation says that the level sets---which are smooth submanifolds by \cite[Theorem A]{Wan}---are equidistant and the second says that they have constant mean curvature. An \emph{isoparametric foliation} $\fol$ on $M$ is a partition of $M$ into the connected components of the level sets of an isoparametric function on $M$. Initiated by work Levi-Civita \cite{LC37} and Cartan \cite{Car38, Car39} in the case of space forms, these types of foliations were studied from different angles, especially for round spheres  (see e.g.\ \cite{Mun80, Mun81, FKM, Abr83,  DN85, Sto, Imm08, Miy, Sif16,  Chi}). Classification efforts were made in the context of symmetric spaces (see e.g.\ \cite{TT95, Kol02, DV16, DG18}) and questions on relations between the topology of $M$ and the existence of isoparametric foliations for some Riemannian metric on $M$  were investigated  (see e.g.\,\cite{GH87, QT15, GT13, Ge16, GR15}).

In this work, we establish the following global finiteness result for Riemannian manifolds $(M,g)$ with an isoparametric foliation $\fol$ in the above sense, up to \emph{foliated diffeomorphism}, i.e.\,diffeomorphism that preserves the decomposition into leaves.

\begin{theorem}\label{main-theorem}
For a fixed dimension $n\neq 5$ and constants $\kappa, \nu, r>0$, the class of isoparametrically foliated closed connected Riemannian manifolds $(M, g, \fol)$ with \[\dim (M)=n,\ \ 
|\pi_1(M)|<\infty,\ \  |\operatorname{sec}(M,g)|<  \kappa,\ \ \operatorname{vol}(M,g)> \nu,\ \  \operatorname{diam}(M,g)<r
\]
contains finitely many foliated diffeomorphism types. In dimension $n=5$, it contains only finitely many foliated homeomorphism types.
\end{theorem}

\begin{remark}\label{remark:first}\ 
\begin{enumerate}[leftmargin=0.65cm]
\item \cref{main-theorem} may be viewed as a foliated version of Cheeger's finiteness theorem \cite{Cheeger,Peters2}. Related results were obtained in \cite{Tapp}  for submersions and in \cite{Harvey} for non-collapsed isometric group actions. Note that Cheeger's result implies that there are only finitely many diffeomorphism classes of manifolds $M$ as in \cref{main-theorem}, so our result can be phrased equivalently as a finiteness result for isoparametric foliations on a \emph{fixed} smooth manifold $M$.
\item The reason why dimension $n=5$ in \cref{main-theorem} is special is that our proof relies on the fact that for a closed $(n-1)$-manifold $L$ the forgetful morphism $\pi_0(\Diff(L))\to \pi_0(\Homeo(L))$ has finite kernel for $n\neq 5$ (see \cref{L:black-magic}). Here $\pi_0(\Diff(L))$ denotes the group of smooth isotopy classes of diffeomorphisms and $ \pi_0(\Homeo(L))$ the group of topological isotopy classes of homeomorphisms. The analogous statement for $n=5$ is known to be false (see \cref{R:when-fail}).
\item Our proof of \cref{main-theorem} goes through more generally for \emph{transnormal functions} which are functions $f:M\to \RR$ that only satisfy the first equation in \eqref{eq: 1}. 
\item We focus on isoparametric foliations of codimension one foliations as defined above in this work, but there are also analogues in higher codimension; see for example \cite{TT95, Lyt14, Tho91}.
\item One may wonder whether the rigid notion of an isoparametric foliation in \cref{main-theorem} can be relaxed, e.g.\ to foliations where the distance between the leaves when measured at different points need not be the same, but can differ by a factor close to $1$. Currently we do not know of a finiteness result in such a  context.
\end{enumerate}
\end{remark}
 
\subsection*{On the assumptions of \cref{main-theorem}}
To illustrate the assumptions in \cref{main-theorem}, we construct various infinite families of isoparametrically foliated Riemannian manifolds. For instance, if the boundedness assumptions on the geometry in  \cref{main-theorem} are dropped, then there are infinite families of mutually inequivalent isoparametric foliations, even on a fixed sphere:

\begin{theorem}\label{T:counter-1}
On $M=\sphere^n$ with $n\ge5$, there are up to foliated diffeomorphism infinitely many isoparametric foliations $(M,g,\fol)$ with singular leaves $\{0\}\times \sphere^1, \sphere^{n-2}\times\{0\}\subset \sphere^n$ and the Clifford torus $\smash{\sphere^1(\tfrac{1}{\sqrt 2})\times \sphere^{n-2}(\tfrac{1}{\sqrt 2})\subset \sphere^n}$ as a regular leaf.
\end{theorem}

\begin{remark}\ 
\begin{enumerate}[leftmargin=0.65cm]
\item A leaf is \emph{singular} if it has codimension $>1$ and \emph{regular} if it has codimension $1$.
\item The bound on $n$ in \cref{T:counter-1} is sharp: by \cite{GR15} there are only finitely foliated diffeomorphism types of isoparametric foliations on simply connected closed $4$-manifolds. The analogous statement in lower dimensions holds too.
\end{enumerate}
\end{remark}

\noindent {There are also infinite families of inequivalent isoparametric foliations on a fixed manifold $M$ that can be distinguished by their singular leaves alone:

\begin{theorem}\label{T:counter-2}
On $M=\sphere^3\times \sphere^4\times \sphere^5$ there are infinitely many isoparametric foliations $(M, g_i,\fol_i)$ that are pairwise not foliated diffeomorphic and can be distinguished by the homeomorphism class of their singular leaves.
\end{theorem}

\noindent An assumption on the fundamental group in \cref{main-theorem} is necessary as well: there are examples of manifolds $M$ that admit infinitely many mutually not foliated diffeomorphic isoparametric foliations, even with respect to a fixed metric $g$ on $M$:

\begin{theorem} \label{T:counter-3}
For $n\geq 4$, there exists a closed Riemannian $n$-manifold $(M,g)$ and infinitely many isoparametric foliations $\fol_i$ on $(M,g)$, such that the leaves of $\fol_i$ for fixed $i$ are all diffeomorphic, but are mutually not homeomorphic for different $i$.
\end{theorem}

\begin{remark}\ 
\begin{enumerate}[leftmargin=0.65cm]
\item The proof of \cref{T:counter-3} shows that the metric $g$ be chosen to be \emph{almost flat}, thus  with arbitrary small curvature and diameter.
\item All leaves of the foliations $\fol_i$ in \cref{T:counter-3} are regular, i.e.\,of codimension $1$. This is in fact the only place, where the assumption on the fundamental group in \cref{main-theorem} is needed: the result remains valid if the assumption  on $\pi_1(M)$ is replaced by requiring that $\fol$ has at least one singular leaf.\end{enumerate}
\end{remark}

\subsection*{A conjectural generalization}
Our strategy of proof for \cref{main-theorem} can be applied to verify the following conjecture for $\dim (X)=1$ (see  \cref{S:prelim} for the relevant definitions).

\begin{conjecture}
For a fixed dimension $n\ge0 $ and constants $\kappa, \nu_1, \nu_2, r  >0$, the class of manifold submetries $P:(M,g)\to X$ where $X$ is an Alexandrov space of volume at least $\nu _2$ and $(M,g)$ is a closed connected Riemannian manifold with
\[\dim (M)=n,\ \  
|\operatorname{sec}(M,g)|<  \kappa,\ \  \operatorname{vol}(M,g)> \nu_1,\ \  \operatorname{diam}(M,g)<r
\]
contains finitely many foliated homeomorphism types.
\end{conjecture} 

\begin{remark}Perelman's Stability Theorem \cite{Kap} implies that in the situation of the conjecture, there are only finitely many base spaces $X$ up to homeomorphism.\end{remark}
\subsection*{Sketch of proofs}We end the introduction by outlining the proofs of our results. 

There are two alternative perspectives on isoparametric foliations on a manifold $(M,g)$: firstly, as the decomposition into the fibers of a \emph{manifold submetry} $P:M\to [c,d]$ to an interval of positive length, and secondly as decompositions of $M$ into two Euclidean disc bundles glued along a diffeomorphisms of their boundary defined in terms of the exponential map of $M$; the foliation is given by the two zero sections and all concentric sphere bundles (see \cref{S:prelim}).

The proof of our main \cref{main-theorem} relies on both of them: the representation as a  manifold submetry $P:M\to [c,d]$ has the advantage that it allows, as $P$ is 1-Lipschitz, for compactness arguments with respect to the Gromov--Hausdorff convergence. The assumption on $\pi_1 (M)$ implies the existence of a singular fiber, and this in turn implies the existence of focal points for regular fibers. A first argument relates the number of focal points within a certain distance both to the length of the base, and to the curvature bounds of the ambient manifold. This implies a uniform lower bound on the length of the interval $[c,d]$ which in turn provides control for a limit submetry of a sequence of manifold submetries $P_i:M_i\to [c_i,d_i]$.
Using this, we argue that in the decompositions of $M_i$ into two Euclidean disc bundles induced by the submetry $P_i:M_i \to [c_i,d_i]$, the disc bundles can be chosen to be \emph{the same} for all $i$, so one is reduced to studying the associated sequence of gluing diffeomorphisms defined in terms of exponential maps. The exponential maps on $M_i$ converge pointwise to the exponential map of a limit manifold, and using this one argues that the associated gluing diffeomorphisms for the disc bundles converge to a homeomorphism. Local contractibility of the group of homeomorphisms of closed manifolds then implies that this sequence is eventually constant up to topological isotopy which implies that the $M_i$ are foliated homeomorphic for large $i$. The upgrade to foliated diffeomorphism is then purely topological. It relies on the fact that there are only finitely many smooth isotopy classes of diffeomorphisms of a closed manifold within a given topological isotopy class.

The foliations in \cref{T:counter-1,T:counter-2,T:counter-3} are constructed as follows: \cref{T:counter-1} is proved by separating $\sphere^n$ as the union of two trivial disc bundles, one over $ \sphere^1$ and one over $ \sphere^{n-2}$, intersecting in a  generalized Clifford torus $\sphere^{n-2} \times \sphere^1$, and gluing them back together along an infinite family of diffeomorphisms of $\mathbb S^{n-2} \times \mathbb S^1$ that are concordant but not isotopic to each other---the existence of such diffeomorphisms is a nontrivial but known result in differential topology. For \cref{T:counter-2}, we argue that $\sphere^3\times \sphere^4 \times \sphere^5$ can be written as the double of total spaces of disc bundles over infinitely many different manifolds. The existence of such bundles has previously been observed in the study of souls of non-negatively curved manifolds. Finally, \cref{T:counter-3} is proved by constructing infinitely many topologically different Riemannian submersions from some nil-manifolds onto circles. 

\subsection*{Acknowledgments}This work originated from discussions following a colloquium talk by MR as part of the Research Training Group DFG 281869850 (RTG 2229), funded by the Deutsche Forschungsgemeinschaft (DFG, German Research Foundation). The authors would like to thank the RTG for providing a fruitful scientific environment. A.L.\ would also like to thank the Hausdorff Research Institute of Mathematics in Bonn where this work was finalised during a stay of A.L.\ as part of the Trimester Program ``Metric Analyis'', funded by the Deutsche Forschungsgemeinschaft (DFG, German Research Foundation) under Germany's Excellence Strategy--EXC-2047/1--390685813. MK was supported by the European Union through an ERC grant (MaFC, 101221003).

\section{Preliminaries}\label{S:prelim}
\noindent Throughout this section, we fix a connected closed smooth manifold $M$. When considering submanifolds of $M$, we allow their components to be of different dimensions.

\subsection{Transnormal functions, manifold submetries, and double disc bundles}\label{SS:srf-ms-ddb}A smooth nonconstant function $f:M\to \mathbb R$ on $M$ equipped with a Riemannian metric $g$ is \emph{transnormal} if it satisfies the first condition in \eqref{eq: 1}, i.e.\,$\|\nabla f\|^2=b(f)$ for a smooth function $b\colon f(M)\ra\RR$. By \cite[Theorem A]{Wan}, the fibers of a transnormal functions gives a partition of $M$ into smooth submanifolds. We will discuss two alternative points of view on these partitions.

The first alternative point of view is that of \emph{manifold submetries} which are maps $P:M\to X$ from $M$ equipped with a Riemannian metric $g$ to a metric space $X$ that are surjective and have the properties that their fibers $P^{-1}(x)\subseteq M$ are pairwise equidistant smooth submanifolds of $M$ with $d_M(P^{-1}(x),P^{-1}(x'))=d_X(x, x')$ for $x,x'\in X$ (c.f.\ \cite{KL22,MR20}). Here $d_M$ is the metric on $M$ induced by $g$, and we write $d_M(A,B)\coloneqq\inf\{d_M(a,b)\mid a\in A, b\in B\}$ for $A,B\subset M$. By \cref{lem:equivalent-chara} below, these are closely related to transnormal functions when the base is an interval $X=[c,d]$ of positive length. In the latter case it follows straight from the definition that a manifold submetry $P:M\to [c,d]$ is given as $P(-)=c+d_M(-,P^{-1}(c))=d-d_M(-,P^{-1}(d))$, so in particular $P$ is always 1-Lipschitz. 

The second alternative point of view is that of \emph{double disc bundles}. To introduce it, we say that \emph{Euclidean disc bundle} $D\ra L$ is the unit disc bundle of a vector bundle equipped with a fiberwise metric over a closed manifold $L$ which may be disconnected and have components of different dimension. A \emph{double disc bundle decomposition} of $M$ consists of two Euclidean disc bundles $D^\pm \ra L^\pm$, a diffeomorphism $\varphi\colon \partial D^+\ra \partial D^-$ between the boundaries of their total spaces, and a diffeomorphism $\tau\colon D^+\cup_{\varphi} D^-\ra M$ from the manifold obtained by gluing $D^\pm$ along the boundaries using $\varphi$. Given such decomposition, there is a canonical map $P_{D^\pm}\colon M\ra [-1,1]$ which is given as the composition of $\tau^{-1}$ with the map $D^+\cup_{\varphi} D^-$ that sends $e\in D^\pm$ to $\pm(\lvert e\rvert-1)$, using the fiberwise norm of $D^\pm$. The fibers over $\pm1$ are the $0$-sections $L^\pm\subset M$ and the fibers over $t\in(-1,1)$ are the concentric sphere subbundles of $D^\pm$.

For the purpose of partitions of $M$ into submanifolds, foliations, transnormal functions, manifold submetries to intervals, and double disc bundles are equivalent:
\begin{lemma}\label{lem:equivalent-chara}
Let $M$ be a connected closed manifold $M$ and $\fol$ a partition of $M$ into smooth submanifolds. The following are equivalent:
\begin{enumerate}[leftmargin=0.6cm]
\item\label{enum:chara-i} $\fol$ arises as the nonempty fibers of a transnormal function $f\colon (M,g)\ra \RR$ for some Riemannian metric $g$.
\item\label{enum:chara-ii} $\fol$ arises as the fibers of a manifold submetry $P\colon (M,g)\ra [c,d]$ with $c<d$ for some Riemannian metric $g$.
\item\label{enum:chara-iii} $\fol$ arises as the fibers of the map $P_{D^\pm}\colon M\ra [-1,1]$ induced by a double disc bundle decomposition $(D^\pm,\tau,\varphi)$ of $M$.
\end{enumerate}
Moreover, the metrics $g$ in \ref{enum:chara-i} and \ref{enum:chara-ii} can be chosen to be the same.
\end{lemma}
\begin{proof}
We first show that \ref{enum:chara-ii} and \ref{enum:chara-iii} are equivalent. Given a manifold submetry $P\colon (M,g)\ra [c,d]$ as in \ref{enum:chara-ii}, we can without changing the fibers rescale the metrics so that $c=-1$ and $d=1$ (up to an opportune translation of the codomain). Now consider the smooth submanifolds $L^\pm\coloneqq P^{-1}(\pm1)$ and the normal exponential map $\exp^\pm\colon \nu_g(L^\pm)\ra M$ from the normal bundle $\nu_g(L^\pm)\subset TM$ of $L^\pm$ with respect to $g$. Using that $P$ is a submetry and that geodesics in $[-1,1]$ with length $<2$ starting from $\pm1$ minimize distance, one sees that the restrictions $\exp^\pm\colon D_g(L^\pm)\ra M$ to the unit disc bundles $D_g(L^\pm)\subset \nu_g(L^\pm)$  are diffeomorphisms onto $P^{-1}([-1,0])$ and  $P^{-1}([0,1])$ respectively (c.f.\,\cite[Section 3.4]{KL22}). Under this identification, the map $P\colon M\ra [-1,1]$ corresponds to $D_g(L^\pm)\ra [-1,1], e\mapsto \pm(\lvert e\rvert-1)$, so we obtain a double disc bundle composition $\smash{D_g(L^+)\cup_{\varphi} D_g(L^-)\cong M}$ with $\smash{P_{D_g(L^\pm)}=P}$ and gluing diffeomorphism $\varphi= (\exp^-)^{-1}\circ \exp^+$ involving  the diffeomorphism $\smash{\exp^\pm\colon \partial D_g(L^\pm) \ra P^{-1}(0)}$ given by restricting the normal exponential map to the unit sphere bundle. This shows that \ref{enum:chara-ii} implies \ref{enum:chara-iii}. To show the converse, one constructs a suitable metric $g$ starting from a double disc bundle decomposition (see \cite[Lemma 3.1, Remark 3.2]{GWY}).

To prove that  \ref{enum:chara-i} and \ref{enum:chara-ii} are equivalent, we use that the fibers of a transnormal function $f\colon M\ra \RR$ as in \ref{enum:chara-i} are pairwise equidistant smooth submanifolds of $M$ by \cite[Theorem A]{Wan}. This implies that there is a canonical metric on $f(M)$ that makes $f\colon M\ra f(M)$ into a manifold submetry. By the intermediate value theorem, $f(M)$ is homeomorphic to a compact interval and thus also isometric to a compact interval with the standard metric. Postcomposing $f$ with this isometry yields a manifold submetry as in \ref{enum:chara-ii} with the same fibers as $f$. Vice versa, given a manifold submetry $P:M\to [-c,d]$, we may assume $[c,d]=[-1,1]$ by rescaling the metric. We choose a smooth and strictly monotonically inreasing function $h\colon [-1,1]\ra \RR$ with $h(-1)=-1$, $h(1)=1$, $h'(t)>0$ on $(-1,1)$, and such that $h(t)=\tfrac{1}{2}(t+1)^2-1$ near $-1$ and $h(t)=1-\tfrac{1}{2}(t-1)^2$ near $1$.  The function $f\coloneq (h\circ P)$ has the same fibers as $P$, so it suffices to check that it is transnormal. As observed above, we have $P(-)=\pm1\mp d_M(-,L^\pm)$. Moreover, we have $(P\circ \exp^{\pm})(v)=\pm1\mp|v|$ for $|v|\leq 2$ and thus $(h\circ P\circ \exp^{\pm})(v)=\pm(1-|v|^2)$ if $|v|$ is close to $0$, so it follows that $f=h\circ P$ is a smooth function. To check that it satisfies the first equation in \eqref{eq: 1}, note that the differential of $f$ vanishes exactly on $L^\pm$. Near $L^{\pm}$ the function $f$ satisfies $\| \nabla f  (\exp^\pm (v))\|^2= \|v\|^2$, which is a linear function of $f$. Outside of $L^{\pm}$,  we have $\| \nabla f\|^2= \|h'  (P(x)) \|^2=\| h' \circ h^{-1}\circ f\|^2$. From this one concludes that $f$ satisfies indeed the first equation in \eqref{eq: 1}, so it is transnormal.\end{proof}

\begin{remark}\label{rem:c-0-bundle-decom}
In the proof of \cref{main-theorem}, we will also rely on a variant of the construction of a double disc bundle decomposition of $M$ out of a manifold submetry $P\colon (M,g)\ra[-1,1]$ from the previous proof where the metric $g$ on the smooth manifold $M$ is only a $\mathcal{C}^1$-metric with a double sided curvature bound in the sense of Alexandrov and the fibers of $P$ are only $\mathcal{C}^1$-submanifolds. In this case, the construction of the double disc bundle decomposition still goes through, but $\tau\colon D(L^+)\cup_{\varphi} D(L^-)\to M$ is now no longer smooth in general and only a homeomorphism, since the exponential map $\exp\colon TM\ra M$ of $g$ is only $\mathcal{C}^0$.
\end{remark}

\subsection{Connectedness of fibers and existence of singular fibers}\label{sec:sing-fib}
Given a manifold submetry $P\colon M\to X$, the base $X$ is an Alexandrov space with curvature bounded below by the lower bound of the sectional curvatures of $M$ and dimension equal to the minimal codimension of the fibers of $P$ (see \cite[Sections 1-3]{KL22}). By \cite[Theorem 12.9]{KL22}, the connected components of the fibers of $P$ are themselves fibers of a manifold submetry $\hat P:M\to \hat X$, which factors $P$ and satisfies $\dim (\hat X)=\dim (X)$. If $\dim(X)=1$, then $X$ is compact, connected $1$-dimensional Alexandrov space, so isometric to an interval or a circle (see \cite[Section 15F]{AKP24}). As the following lemma shows, which of the two it is turns out to be related to whether $P$ has at least one \emph{singular fiber}, i.e.\,a fiber with a component of codimension $\ge 2$. Note that if $X$ is an interval, then no fiber over points in the interior can be singular (this follows for instance from the double disc bundle description in \cref{lem:equivalent-chara}).
\begin{lemma}\label{lem:sing-fib}Fix a manifold submetry $P\colon (M,g)\ra X$ from a connected closed Riemannian manifold $(M,g)$ to an Alexandrov space $X$ with $\dim(X)=1$.
\begin{enumerate}
\item\label{enum:sing-fib-i}  If $P$ has at least one singular fiber, then  $X$ is isometric to an interval.
\item\label{enum:sing-fib-ii} If $M$ has finite fundamental group, then $X$ is isometric to an interval and $P$ has at least one singular fiber. \end{enumerate}
\end{lemma}
\begin{proof}
As explained above, $X$ is isometric to an interval or to a circle (see \cite[Section 15F]{AKP24}). If it is a circle, then $P$ is a $\mathcal{C}^1$-Riemannian submersion by \cite[Theorem A]{BG00}, so by the $\mathcal{C}^1$-version of Ehresmann fibration theorem (see e.g.\,\cite[Proposition, p.\,39]{EarleEells}) it is in particular a topological fiber bundle. Hence all fibers of $P$ are homeomorphic, so have the same dimension and thus are non-singular. This shows \ref{enum:sing-fib-i}. For \ref{enum:sing-fib-ii}, we first argue that $|\pi_1(M)|<\infty$ implies that $X$ is isometric to an interval. This follows by a similar argument as for \ref{enum:sing-fib-i}: if $X$ is a circle, then $P:M\to S^1$ is a fiber bundle. Since the fibers have finitely many components by compactness, $\mathrm{Im}(P_*)\le \pi_1(S^1)\cong\ZZ$ has finite index, so $\pi_1(M)$ is infinite. Using \cref{lem:equivalent-chara}, in order to prove \ref{enum:sing-fib-ii}, it thus suffices to show that there is no double disc bundle decomposition $M\cong D(L^+)\cup_{\varphi} D(L^-)$ with $L^\pm\subset M$ of codimension $1$ if $M$ has finite fundamental group. This is proved for instance in \cite[Proposition 3.3]{De23}. 
\end{proof}

\begin{corollary}\label{cor:isoparam-char}In the situation of \cref{lem:equivalent-chara}, if the submanifolds in $\fol$ are connected and there is at least one singular leaf, then \ref{enum:chara-i}-\ref{enum:chara-iii} are also equivalent to
\begin{enumerate}[resume]
\item[(iv)]\label{enum:chara-iv} $\fol$ arises as an isoparametric foliation on $M$with respect to some Riemannian metric on $g$ in the sense of the introduction.
\end{enumerate}
\end{corollary}
\begin{proof}
An isoparametric function $f\colon M\ra \RR$ is in particular transnormal, so by \cref{lem:equivalent-chara} there is manifold submetry $P\colon M\ra [c,d]$ with the same fibres. Its connected components are the fibres of another submetry $\hat P\colon M\ra \hat X$ whose base is still an interval because of \cref{lem:sing-fib}. This shows that (iv) implies \ref{enum:chara-ii}. Conversely, the fact that \ref{enum:chara-i} implies (iv) follows from \cite[Corollary D]{LiYang}.
\end{proof}

\subsection{Foliated diffeomorphism, concordance, and isotopy}\label{S:conc-iso-diffeo}
In \cref{SS:srf-ms-ddb}, we considered partitions $\fol$ of a connected closed manifold $M$ into submanifolds induced by a transnormal function, a submetry, or a double disc bundle decomposition. Two such partitioned manifolds $(M,\fol)$ and $(M',\fol')$ are \emph{foliated diffeomorphic (homeomorphic)} if there is a diffeomorphism (homeomorphism) that sends leaves to leaves. In the case of double disc bundles, this can be rephrased in terms of the gluing diffeomorphisms $\varphi$. To make this precise, we recall some terminology:

Two diffeomorphisms (homeomorphisms) $\varphi_0,\varphi_1\colon Q\ra R$ between closed manifolds $Q$ and $R$ are \emph{smoothly (topologically) concordant} if $(\varphi_0\sqcup \varphi_1)\colon Q\times\{0,1\}\ra R\times\{0,1\}$ extends to a diffeomorphism (homeomorphism) $\phi \colon Q\times [0,1] \ra R\times [0,1]$. If $\phi$ can be chosen so that $\phi(Q\times\{t\})\subset R\times\{t\}$ for $t\in[0,1]$, then the $\varphi_i$ are called \emph{smoothly (topologically) isotopic}. We write $\tilde{\pi}_0(\Diff(Q,R))$ for the set of smooth concordance classes of diffeomorphisms, ${\pi}_0(\Diff(Q,R))$ for the set of smooth isotopy classes of diffeomorphisms, and similarly $\tilde{\pi}_0(\Homeo(Q,R))$ and ${\pi}_0(\Homeo(Q,R))$ for homeomorphisms. As the notation suggests, the sets ${\pi}_0(\Diff(Q,R))$ and ${\pi}_0(\Homeo(Q,R))$ are the sets of path-components of spaces, namely the spaces of diffeomorphisms and homeomorphisms equipped with the $\mathcal{C}^\infty$- or the compact-open topology respectively. 

\begin{proposition}\label{P:foliated}
Fix manifolds $(M_0,\fol_0)$ and $(M_1,\fol_1)$ with partitions induced by 
double disc bundle decompositions $M_i\cong D^+\cup_{\varphi_i} D^-$ as in \cref{SS:srf-ms-ddb} with the same  bundles $D^\pm\ra L^\pm$ but different gluing diffeomorphisms $\varphi_i\colon \partial D^+\to \partial D^-$, $i=0,1$.
\begin{enumerate}[leftmargin=0.8cm]
\item\label{equ:foliated-isotopy} If $\varphi_0, \varphi_1$ are smoothly (topologically) isotopic, then $(M_0,\fol_0)$ and $(M_1,\fol_1)$ are foliated diffeomorphic (homeomorphic).
\item\label{equ:foliated-concordance} If $\varphi_0, \varphi_1$ are smoothly (topologically) concordant, then there is a diffeomorphism (homeomorphism) $M_0\cong M_1$ that sends $L^\pm$ and $L^0=\partial D^\pm$ to themselves.
\end{enumerate}
\end{proposition}
\begin{proof} It is convenient to replace $D^+\cup_{\varphi_i} D^-$ up to foliated diffeomorphism with $\smash{D^+\cup_{\id_{\partial D^+}}\big(\partial D^+\times [0,1]\big)\cup_{\varphi_i}D^-}$. Here the latter is equipped with the partition into submanifolds given as the fibers of the map to $[-1,2]$ that sends $e\in D^+$ to $\lvert e\rvert-1$, $e\in D^-$ to $2-\lvert e\rvert$, and agrees with the projection  to $[0,1]$ on the middle part. If the $\varphi_i$ are smoothly isotopic, then $\id_{\partial D^+}$ and $\varphi_0^{-1}\circ \varphi_1$ are smoothly isotopic. A choice of isotopy $\phi\colon \partial D^+\times [0,1]\to \partial D^+\times [0,1]$ between them induces a diffeomorphism \[\hspace{-0.35cm}(\id_{D^+}\cup\phi\cup\id_{D^-})\colon \Big(D^+\cup_{\id_{\partial D^+}}\big(\partial D^+\times [0,1]\big)\cup_{\varphi_0}D^-\Big)\to \Big(D^+\cup_{\id_{\partial D^+}}\big(\partial D^+\times [0,1]\big)\cup_{\varphi_1}D^-\Big)\]
which is compatible with the maps to $[0,2]$, so is foliated. If $\phi$ is merely a concordance, then the above diffeomorphism still commutes with maps to $[0,2]$ except on the interior of the part $\partial D^+\times (0,1)$, so in particular it preserves the fibers over $0,1,2$, which are $L^+$, $L^0$, and $L^-$ respectively. This implies the claim in the smooth cases. The proof in the topological cases is analogous.
\end{proof}

The converse in \cref{P:foliated} \ref{equ:foliated-isotopy} does not hold in general, but the failure can be described in terms of the topological group $\Iso_b(D^\pm)$ of self-diffeomorphisms of $D^\pm$ that cover a self-diffeomorphism of $L^\pm$ and are fiberwise isometries, equipped with the $\mathcal{C}^\infty$ topology. Restricting such diffeomorphisms to the boundary sphere bundle yields a morphism $\pi_0(\Iso_b(D^\pm))\ra \pi_0(\Diff(\partial D^\pm))$ whose image we denote by $H_\pm\le \pi_0(\Diff(\partial D^\pm))$. These subgroups admit commuting actions on $\pi_0(\Diff(\partial D^+,\partial D^-))$; the group $H_-$ acts by postcomposition and $H_+$ by precomposition with the inverse. As observed in \cite[Corollary 2.5]{Ge16}, given a diffeomorphism $\varphi\colon \partial D^+\ra \partial D^-$, the foliated diffeomorphism type of the associated double disc bundle is encoded by the class of $\varphi$ in the double cosets of these commuting actions of $H_\pm$: 

\begin{proposition}\label{P:fol-diff}
In the situation of \cref{P:foliated}, the double disc bundles $M_i=D^+\cup_{\varphi_i}D^-$ are foliated diffeomorphic if and only if the gluing diffeomorphisms $\varphi_i$ for $i=0,1$ lie in the same double coset of $H_-\backslash \pi_0(\Diff(\partial D^+, \partial D^-))/H_+$.\end{proposition}

\section{Proof of the main result}\label{S:proof}
This section serves to prove the following finiteness result on manifold submetries:

\begin{theorem}\label{T:true-main}
For a fixed dimension $n$ and constants $\kappa, r,\nu>0$,  consider the class $\mathcal M^{n,\kappa,r,\nu}$ of  manifold submetries $P:M\to X$ with connected fibers out of closed connected Riemannian manifolds $M$ such that $P$ has at least one singular fiber and 
\[\dim(X)=1,\ \ \dim(M)=n,\ \   |\sec(M)|<\kappa,\ \  \operatorname{vol}(M)>\nu,\ and\  \operatorname{diam}(M)<r.
\]
Then $\mathcal M^{n,\kappa,r,\nu}$  contains only finitely many foliated homeomorphism types. If $n\neq 5$, then it  also contains only finitely many foliated diffeomorphism types.
\end{theorem}

Note that \cref{T:true-main} implies \cref{main-theorem} in view of the characterisation of isoparametric foliations on $M$ with $|\pi_1(M)|<\infty$ in terms of submetries to intervals with singular fibers resulting from \cref{lem:equivalent-chara,lem:sing-fib,cor:isoparam-char}. 

\subsection{Proof of \cref{T:true-main}}We begin with the following non-collapsing result:
\begin{proposition}\label{P:int-bound}
Fix constants $\kappa,\ell>0$ and a manifold submetry $\subm:M^n\to  [-\ell,\ell]$ on a closed connected Riemannian manifold $(M,g)$ with connected fibers and at least one singular fiber. If $\sec(M, g)<\kappa$, then $\ell\geq{\pi\over 4n\sqrt{\kappa}}$.
 \end{proposition}
\begin{proof}
We may assume without loss of generality that $L^+\coloneq P^{-1}(\ell)$ is singular, i.e.\,of codimension $>1$. We will rely on the fact that any geodesic $\gamma$ in $M$ starting at $L^0$ with $\gamma'(0)\perp L^0$ and $|\gamma'(0)|=1$ is projected by $P$ to a piecewise geodesic in $[-l,l]$ parametrized by arclength whose break points occur only at $-l$ and $l$ (see \cite[Propositions 3.2 and 12.5, Corollary 7.4]{KL22}). Thus if $\gamma$ starts in the direction of $L^+$, then $\gamma((1+4k)\cdot \ell)\in L^+$ for all $k\ge0$. Since $L^+$ has strictly smaller dimension than $L^0$, this implies that the points $\gamma((1+4k)\cdot \ell)$ for $k\ge0$ are singular values for the normal exponential map $\exp\colon \nu_g(L^0)\ra M$, i.e.\,that they are focal points of $L^0$. The number of focal points for $L^0$ on $\gamma\colon (0,4n\ell)\ra M$, counted with multiplicity, is thus at least $n$. Now suppose by contradiction that $\smash{4n\ell<{\tfrac{\pi}{\sqrt \kappa}}}$. It then follows from Rauch's comparison theorem that $\gamma\colon (0,4n\ell)\ra M$ has no conjugate points. But this implies by \cite[Corollary 1.2]{Lyt09} that the number of focal points of $L^0$ along $\gamma\colon (0,4n\ell)\ra M$ is at most $\mathrm{dim}(L^0)=n-1$ (this result is stated for manifolds without focal points, but the argument applies to any geodesic segment without conjugate points), so we arrive at a contradiction.
\end{proof}

Equipped with this proposition, we now prove  \cref{T:true-main}. Arguing by contradiction, we choose an infinite sequence $(M_i, g_i, \subm_i)\in \mathcal M^{n,\kappa,r,\nu}$ of manifold submetries that are pairwise not foliated diffeomorphic, and if $n\neq5$ also pairwise not foliated homeomorphic. By \cref{lem:sing-fib}, the base $X_i$ is isometric to an interval, so up to postcomposition by an isometry, we may assume $X=[-l_i,l_i]$ with 
$2 l_i \leq\operatorname{diam}(M_i)<r$. By \cref{P:int-bound}, there is a common positive lower bound for all the $l_i$, so by rescaling by the uniformly bounded numbers $\frac 1 {l_i}$ and changing $\kappa, r, \nu$, we may assume $l_i=1$ so $X=[-1,1]$ for all $i$. By Cheeger--Gromov's compactness theorem \cite[Theorem 4.4]{Peters},
upon choosing a subsequence we can find a smooth manifold $M$ and diffeomorphisms $\psi_i:M\to M_i$,  such that the metrics $\psi_i^*(g_i)$ $\mathcal C{}^{1}$-converge to a $\mathcal C{}^{1}$-metric $g_\infty$ on $M$ (in fact to a $\mathcal C{}^{1+\alpha}$-metric for any $\alpha\in(0,1)$ but we will not need this). Moreover, the exponential maps $\exp_i\coloneqq (\psi_i^{-1}\circ \exp_{(M_i,g_i)}\circ d\psi_i):TM\to M$ converge uniformly on compact subsets to $\exp_\infty\coloneqq \exp_{(M,g_\infty)}\colon TM\to M$. Identifying $M_i$ with $M$ via $\psi_i$, we may assume that $M_i=M$ and $\psi _i=\mathrm{id}_M$ for all $i$. By \cite[Lemma 2.4]{KL22}, upon passing to a further subsequence, we can furthermore assume that the submetries $P_i: (M,g_i)\ra [-1,1]$ converge pointwise to a submetry $P_\infty:(M,g_\infty)\to [-1,1]$ such that the fibers $P_i^{-1}(t)$ converge pointwise to the fibers $P_\infty^{-1}(t)$ for all $t$. The $\mathcal C{}^{1}$-Riemannian manifold $(M,g_\infty)$ inherits the double-sided curvature bound in the sense of Alexandrov from the $(M,g_i)$ (cf.\ \cite{KL21})  and all fibers of $P_\infty$ are $\mathcal C^{1}$-submanifolds of $M$ by \cite[Corollary 12.6]{KL22}. We write
\[L_i^{\pm}\coloneqq P_i^{-1}(\pm 1)\quad \text{and}\quad L_i^{0}\coloneqq P_i^{-1}(0)\quad \text{for } i\le \infty.
\]
By \cref{lem:equivalent-chara} and its proof, the foliated diffeomorphism type of $P_i$ for $i<\infty$ is that of the double disc bundle $D_{g_i}(L^+_i)\cup_{\varphi_i} D_{g_i}(L^-_i)$ where $\varphi_i$ is the gluing diffeomorphism given as the composition \begin{equation}\label{equ:gluing-diffeo}\varphi_i\coloneqq \big((\exp_i^-)^{-1}\circ \exp_i^+\big)\colon 
\partial D_{g_i}(L^+_i)\ra \partial D_{g_i}(L^-_i)\end{equation} where $\exp_i^\pm\colon \partial D_{g_i}(L^+_i)\ra L^0_i$ is the restriction of the normal exponential map of $(M,g_i)$. Moreover, the same holds for the foliated \emph{homeomorphism} type in the case $i=\infty$ (see \cref{rem:c-0-bundle-decom}), but we do \emph{not} get the latter for the foliated $\mathcal C^{1}$ diffeomorphism type. This is because the exponential map $\exp$ is no better than Lipschitz continuous in general and the smooth exponential maps $\exp_i$ need not converge in the $\mathcal C^1$-sense (cf.\ \cite{Pu}). The strategy is now to show that we can identify the submanifolds $L^\pm_i$ and their normal bundles for different $i$ with each other and thus reduce to studying convergence of the gluing maps $\varphi_i$ on a single space. The main ingredient to achieve this is the following proposition:

\begin{proposition}\label{prop:aux-prop}
Upon passing to a subsequence, there are smooth submanifolds $\tilde L^\pm\subset M$ and diffeomorphisms $\alpha_i\colon M\ra M$ which converge in the $\mathcal C{}^{1}$-topology to a $\mathcal C{}^{1}$-diffeomorphism $\alpha_\infty$ such that $\alpha_i(\tilde L^\pm)=L_i^\pm$ for all $i\le \infty$.
\end{proposition}

Before turning to the proof of this proposition, we use it to finish the proof of \cref{T:true-main}.  The derivatives of the $\alpha_i$ induce smooth bundle isomorphisms $d\alpha_i\colon \nu(\tilde{L}^\pm)\ra  \nu(L_i^\pm)$ between the normal bundles of the $\tilde{L}^\pm$ and $L_i^\pm$ as quotient vector bundles. Identifying the quotient bundle $\nu(L_i^\pm)$ with the subbundle $\nu_{g_i}(L_i^\pm)\subset TM$ using the Euclidean metric $g_i$ and pulling back $g_i$ along $d\alpha_i$, we obtain a sequence of smooth Euclidean metrics $d\alpha_i^*g_i$ on the vector bundle $\nu(\tilde{L}^\pm)$ that converge against a $\mathcal{C}^0$-Euclidean metric $d\alpha_\infty^*g_\infty$. Now fix a smooth Euclidean metric $\tilde g$ on $\nu(\tilde{L}^\pm)$ and choose smooth Euclidean bundle isomorphisms $u_i\colon (\nu(\tilde{L}^\pm),\tilde g)\ra (\nu(\tilde{L}^\pm),d\alpha_i^*g_i)$ that cover the identity and converge uniformly on compact subsets against a $\mathcal{C}^0$ isomorphisms $u_\infty\colon (\nu(\tilde{L}^\pm),\tilde g)\ra (\nu(\tilde{L}^\pm),d\alpha_\infty^*g_\infty)$ (this is possible since Euclidean bundles with respect to different metrics are isomorphic and these isomorphisms can be chosen to be close if the metrics are close). With respect to the resulting isomorphisms $w_i\coloneqq (d\alpha_i\circ u_i)\colon (\nu(\tilde{L}^\pm),\tilde g)\ra (\nu_{g_i}(L_i^\pm),g_i)$, the gluing diffeomorphism \eqref{equ:gluing-diffeo} becomes a diffeomorphism between the sphere bundles of $(\nu(\tilde L^\pm),\tilde g)$,
\begin{equation}\label{equ:new-gluing-diffeo}\smash{\big(w_i^{-1} \circ \varphi_i\circ w_i)\colon S(\nu(\tilde L^+),\tilde g)\ra  S(\nu(\tilde L^-),\tilde g)},\end{equation} so now source and target are the same for all $i$. As the $d\alpha_i$, and $w_i=d\alpha_i\circ u_i$ converge uniformly against $\alpha_\infty$ and $w_\infty= d\alpha_\infty\circ u_\infty$ respectively, the diffeomorphisms \eqref{equ:new-gluing-diffeo} converge uniformly against the homeomorphism $\smash{w_\infty^{-1} \circ \varphi_\infty\circ w_\infty}$. Using that the space of homeomorphisms between two closed manifolds (in our case the (co)domain of \eqref{equ:new-gluing-diffeo}) with the topology of uniform convergence is locally contractible by \cite{Cern69}, it follows that the diffeomorphisms \eqref{equ:new-gluing-diffeo} are all topologically isotopic to each other for $i\gg0$, so \cref{P:foliated} implies that the $(M,g_i,P_i)$ are foliated homeomorphic to each other for $i\gg0$. Moreover, if $n\neq5$, then by \cref{L:black-magic} below there are only finitely many smooth isotopy classes of diffeomorphisms between closed manifolds within a given topological isotopy class, so it follows that up to passing to a subsequence the diffeomorphisms $\tilde \varphi_i$ are all smoothly isotopic to each other, and therefore the $(M,g_i,P_i)$ are all foliated diffeomorphic by \cref{P:foliated}. This is a contradiction to our choice of the $(M,g_i,P_i)$.

\begin{proof}[Proof of \cref{prop:aux-prop}]It suffices to construct
 $\alpha_i$ as $\mathcal{C}^1$-diffeomorphisms, since the space of $\mathcal{C}^\infty$-diffeomorphisms of $M$ that send $\tilde L^\pm$ to $L^\pm_i$ is dense in the space of $\mathcal{C}^1$-diffeomorphisms with this property (see e.g.\cite[2.2.7, 3.3.5]{Hir94}).

The smooth submanifolds $L^{\pm}_i\subset (M,g_i)$ have normal  injectivity radius at least $1$ (hence, these subsets have \emph{reach} at least $1$ in $(M,g_i)$, cf. \cite{KL21}, \cite{Lyt05}).
Denote by $\Pi_i$ the closest-point projection onto $L^+_i$ with respect to the metric  $g_i$, defined in a tube of radius $1$ around $L_i^+$. Similarly, let $\Pi_\infty$ be the closest-point projection onto the $\mathcal{C}^1$-submanifold $L_\infty^+\subset (M,g_\infty)$.  Note, that $\Pi_\infty$ need not be $\mathcal C^1$. Now fix a small $\delta >0$.
For all sufficiently small $\varepsilon=\varepsilon(\delta)>0$,
 the projections $\Pi_i$ and $\Pi_\infty$ are $(1+\delta)$-Lipschitz continuous on the $4\varepsilon$-tubes around $L_i^+$ \cite[Theorems 1.6, 1.3]{Lyt05}.  For all sufficiently large $i$,  the Hausdorff-distance between $L_{\infty}^+$ and $L_i^+$ is at most $\varepsilon$.
Then the composition $(\Pi_i \circ \Pi_\infty)\colon L_i^+\to L_i^+$ is a continuous map $2\varepsilon$-close to the identity.  Once $2\varepsilon$ is smaller than the  injectivity radius of $M$,  we consider the homotopy  along shortest geodesics inside $M$ from $\Pi_i \circ \Pi_\infty$ and the identity of $L_i^+$.  Projecting this homotopy by $\Pi_i$ to $L_i^+$ we deduce that $(\Pi_i \circ \Pi _{\infty})$ is homotopic to the identity of $L_i^+$ inside $L_i^+$.  Therefore, $(\Pi_i \circ \Pi_\infty):L_i^+\to L_i^+$ is surjective and thus $\Pi_i\colon L_{\infty}^+\ra L_i^+$ is surjective  too.
 Exchanging the roles of $\Pi_\infty$ and $\Pi_i$, we see that $\Pi_\infty:L_i^+\to L_\infty^+$ is also surjective. So $L_\infty^+$ and $L_i^+$ have the same dimension and almost the same volume (up to the multiplicative constant $(1+\delta)^{n-1}$, which can be made arbitrary close to $1$). 
 
 Now fix points $x_i \in L_i ^+$ converging to a point $x_\infty\in L_\infty^+$ and geodesics $\gamma _i$ in $(M,g_i)$ starting at $x_i$ that are orthogonal to $L_i^+$.  Then $\gamma _i$ are horizontal with respect to $P_i$ and converge  to a horizontal geodesic $\gamma_\infty$  in $(M,g_\infty)$ starting  at $x_\infty$, cf. \cite[Section 12]{KL22}.  We deduce that the normal
spaces of $\nu _{x_i}L_i^+$ with respect to the metric $g_i$ converge to a subspace of the normal space $\nu_{x_\infty} L_\infty^+$ with respect to $g_\infty$. Since the dimensions are equal,  $\nu_{x_i} L_i^+$ converge to $\nu_{x_\infty}L_\infty^+$, and therefore also the tangent spaces $T_{x_i}L_i^+$ converge to $T_{x_\infty}L_\infty^+$.

To construct the desired diffeomorphisms $\alpha_i$, we first construct for all $i$ large enough $\mathcal C^1$-diffeomorphisms $\eta_i^+:M\to M$ that converge in the $\mathcal C^1$-topology to the identity, satisfy $\eta_i^+ (L_i^+)=L_\infty^+$ and have support in a small neighborhood of $L_\infty^+$. To this end, we first change coordinates:  we fix a $\mathcal C^1$-diffeomorphism  $\zeta :M\to M$ supported in an arbitrarily small neighborhood of $L_\infty^+$ such that $\bar L^+:=\zeta (L_\infty^+)$ is smooth (see e.g.\,\cite[Theorem 3.6]{Hir94} for the fact that such a $\zeta$ exists). Then the $\mathcal C^1$-submanifolds $\bar L_i^+:=\zeta (L^+_i)$ converge  in the Hausdorff-sense to $\bar L^+$, their tangent spaces at points in $\bar L_i^+$ converge to the corresponding tangent space at points in $\bar L^+$, and the volumes of $\bar L_i^+$ converge to the volume of $\bar L^+$. Now we smoothly identify a small tubular neighborhood of $\bar L^+$   with the smooth normal bundle. It follows that the foot-point projection of $\bar L_i^+$ to $\bar L^+$ is 
a $\mathcal C^1$ immersion, hence a covering map  for all large $i$.  Since this projection is almost $1$-Lipschitz, and the volumes converge, we deduce that this projection is injective.  Therefore, $\bar L_i^+$ can be uniquely  represented as a $\mathcal C^1$-section $h_i$ of the normal bundle of $\bar L^+$.  The convergence statement for the tangent spaces, implies that this sections  $h_i$ converge to the $0$-section in the $\mathcal C^1$-topology.  Using the sections $h_i$, we find $\mathcal{C}^1$-diffeomorphisms of the normal bundle that send $h_i$ to the $0$-section, are supported in a small neighbourhood of the $0$-section, and converge in the $\mathcal C^1$-topology. Extending these diffeomorphisms by the identity outside of the tubular neighbourhood of $\bar L^+$ and changing back coordinates using $\zeta^{-1}$, we obtain diffeomorphisms $\eta_i^+$ as claimed. Applying the same reasoning to $L_\infty^-$ and composing the resulting diffeomorphisms, we find diffeomorphisms $\eta_i:M\to M$ which converge in the $\mathcal C^1$-topology to the identity and satisfy
$\eta_i (L_i^\pm)=L_\infty^\pm$. Finally, we postcompose by a choice of a $\mathcal C^1$-diffeomorphism  supported in a neighborhood of  $L_\infty^{\pm}$ in order to move $L_\infty^{\pm}$ to smooth submanifolds $\tilde L^{\pm}$ of $M$ (using e.g.\,\cite[Theorem 3.6]{Hir94}). Taking the inverse of the resulting $\mathcal C^1$-diffeomorphism yields diffeomorphisms $\alpha_i$ as claimed.\end{proof}

\section{Examples of infinite families of foliations}\label{S:counter}
This section serves to prove \cref{T:counter-1,T:counter-2,T:counter-3}.

\begin{proof}[Proof of \cref{T:counter-1}]We use the perspective on isoparametric foliations on simply connected manifolds in terms of double disc bundles provided by \cref{lem:equivalent-chara,lem:sing-fib,cor:isoparam-char}. Consider the trivial disc bundles $D^+\coloneqq \DD^2\times \sphere^{n-2}\to \sphere^{n-2}$ and $D^-\coloneqq\sphere^1\times \DD^{n-1}\to \sphere^1$ equipped with the standard Euclidean metric. We identify $D^+\cup_{\mathrm{id}}D^-$ with $\sphere^n$ via the standard diffeomorphism that sends the zero-sections $L^+\subset D^+$ and $L^-\subset D^-$ to $\sphere^{n-2}\times\{0\}\subset \sphere^n$ and $\sphere^{1}=\{0\}\times \sphere^{1}\subset \sphere^n$ respectively, and the boundary $L^0=\partial D^\pm=\sphere^{n-2}\times \sphere^1$ to  ${\sphere^{n-2}({\sfrac{1}{ \sqrt{2}}})\times \sphere^{1}({\sfrac{1}{ \sqrt{2}}})\subset \sphere^n}$. Since $n\ge5$, by \cref{lem:diffeos1sn} below, there is an infinite family of mutually non-isotopic diffeomorphisms $\varphi_i\in \Diff(\partial D^+,\partial D^-)$ that are concordant to the identity. By \cref{P:foliated} \ref{equ:foliated-concordance}, there are diffeomorphisms between the glued manifolds $M_i\coloneqq D^+\cup_{\varphi_i}D^-$ and $\sphere^{n}=D^+\cup_{\mathrm{id}}D^-$ that preserve the leaves $L^\pm$ and $L^0$, so it suffices to show that infinitely many of the $M_i$s are mutually not foliated diffeomorphic with respect to the standard foliation induced by the double disc bundle structures. This is by \cref{P:fol-diff} equivalent to showing that infinitely many of the $\varphi_i$ yield different classes in the double coset $H_-\backslash \pi_0(\Diff(\partial D^+, \partial D^-))/H_+$ for which it is enough to show that the groups $H_{\pm}=\mathrm{im}(\pi_0(\operatorname{Isom}_b(D^\pm))\ra \pi_0(\Diff(\partial D^\pm))$ are finite. To show this, we consider the fibration sequences
\begin{equation}\label{equ:fibrations-isom-groups}
\begin{tikzcd}[column sep=0.5cm, row sep=0.0cm]
\operatorname{Map}(\sphere^{n-2}, \oO(2))\rar& \operatorname{Isom}_b(D^+)\rar&\Diff(\sphere^{n-2})\\
\operatorname{Map}(\sphere^{1}, \oO(n-1))\rar& \operatorname{Isom}_b(D^-)\rar& \Diff(\sphere^{1})
\end{tikzcd}
\end{equation}
 induced by restricting to the diffeomorphisms on the base. For $m\ge1$, we have $\pi_0(\operatorname{Map}(\sphere^{m}, \oO(k)))=\pi_m(\oO(k))\rtimes \ZZ/2$, so as $\pi_m(O(2))=\pi_m(\sphere^1)=0$ for $m\ge2$ and  $\pi_1(O(m))=\ZZ/2$ for $m\ge3$, we see that the two mapping spaces in \eqref{equ:fibrations-isom-groups} have finitely many components. Since $\pi_0(\Diff(\sphere^{m}))$ is finite for $m\neq4$ by  \cref{lem:finitenessdiffeosphere} below, it follows from the long exact sequence in homotopy groups that $\pi_0(\operatorname{Isom}_b(D^\pm))$ and thus $H_{\pm}$ is finite for $n\neq6$. For $n=6$, the group $\pi_0(\Diff(\sphere^{n-2}))$ maybe be infinite, but since $\pi_0(\operatorname{Map}(\sphere^{n-2}, O(2)))$ is finite even for $n=6$ and $\partial D^{\pm}=\sphere^{n-2}\times  \sphere^1$, for finiteness of $H_{+}$ it actually suffices that that the image of  $(-)\times \sphere^1\colon \pi_0(\Diff(\sphere^{n-2}))\ra \pi_0(\Diff(\sphere^{n-2}\times  \sphere^1))$ is finite for $n=6$ which holds by \cref{lem:imagefinitenessdiffeo4sphere} below.
\end{proof}

\begin{proof}[Proof of \cref{T:counter-2}]We again use the perspective on isoparametric foliations on simply connected manifolds in terms of double disc bundles provided by \cref{lem:equivalent-chara,lem:sing-fib,cor:isoparam-char}. By \cite{Bel03} there are infinitely many metrics $g_i$ on $N\coloneqq \sphere^3\times \sphere^4\times \RR^5$ with non-negative sectional curvature, with mutually non-homeomorphic \emph{souls} $\Sigma_i$; these are closed submanifolds  of $N$ whose normal bundles $\nu_i\coloneqq \nu_{g_i}(\Sigma_i)$ are diffeomorphic to $N$. Since the unit sphere bundle $\nu_i^{=1}$ is simply connected, \cref{lem:interior-unique} below ensures that there is a diffeomorphism $\varphi_i\colon \nu_i^{\le 1}\ra \sphere^3\times  \sphere^4\times \DD^5$. The double disc bundles $\smash{M_{i,j}\coloneqq  \nu^{\le 1}_{i}\cup_{\partial(\varphi_j^{-1}\circ\varphi_i)} \nu^{\le 1}_{j}}$ are thus diffeomorphic to $\sphere^3\times \sphere^4\times \sphere^5$ for all $i,j$ and the induced foliation has singular leaves $\Sigma_i$ and $\Sigma_j$ by construction. 
\end{proof}

\begin{remark}
Applying \cite{Gui98}, the  exotic double disc bundle structures on  $\sphere^3\times \sphere^4\times \sphere^5\cong M_{i,j}$ from the previous proof can be seen for $i=j$ to arise from submetries $P:(M_{i,i},g)\to[-1,1]$ with respect to a non-negatively curved metric $g$ on $M_{i,i}$.
\end{remark}

\begin{proof}[Proof of \cref{T:counter-3}]
Denote the Heisenberg group of upper triangular $3\times 3$ real matrices with $1$ on the diagonal by $\mathcal H$ and consider the homomorphism $h:\mathcal H\to \RR$ sending $A=(a_{ij})$ to $a_{12}$. For integers $m\ge 1$, we write $H_m\le \mathcal H$ for the subgroup of integers matrices $A$ such that $h(A)$ is divisible by $m$. The group $H_1$ is the standard lattice in $\mathcal H$, the integer Heisenberg group, and $H_m\le H_1$ is a normal subgroup of index $m$, so also uniform lattice in $\mathcal H$. The abelianization of $H_m$ is isomorphic  to $\mathbb Z^2\oplus \mathbb Z/m$, so in particular the groups $H_m$ are mutually non-isomorphic.

We consider the manifold $M= \sphere ^1\times H_1\backslash \mathcal H$, whose universal covering $\tilde M$ is the nilpotent Lie group group $\mathbb R\times \mathcal H$. For all natural $m$, we consider the homomorphisms $\Phi_m:\mathcal H\to \RR \times \mathcal H,A\mapsto(\frac 1 m\cdot h(A), A)$.   The image $\mathcal H_m:=\Phi_m (\mathcal H)$ is  normal in $\mathbb R \times \mathcal H$. 

Fix a left-invariant Riemannian metric on $\tilde M$  and equip $M$ with the corresponding locally homogeneous metric. 
 The orbits of the left multiplication by $\mathcal H_m$ define a simple foliation of $\tilde M$ by $3$-dimensional translates of $\mathcal H_m$.  By the left-invariance of the metric, this foliation is Riemannian and each orbit has constant mean curvature. Since $\mathcal H_m\le \RR\times \mathcal H$ is normal, these orbits are permuted by the action of the larger group $\RR \times \mathcal H$ of isometries of $\tilde M$.  Therefore, the constant mean curvature of the orbits is not only constant on each orbit but this constant does not depend on the orbit.  Moreover, this simple foliation projects to a foliation $\mathcal F_m$ on $M$.

 Setting $K_m:=\mathcal H_m \cap (\ZZ\times  H_1)$, all fibers of $\mathcal F_m$ are diffeomorphic to $K_m\backslash \mathcal H_m= \Phi_m^{-1} (K_m)\backslash \mathcal H =  \Phi_m^{-1} (\ZZ\times H_1)\backslash \mathcal H.$ Moreover, by definition of $\Phi_m$, we see $\Phi_m^{-1} (\ZZ\times H_1)=H_m$. Hence, the fibers of $\mathcal F_m$ are diffeomorphic to  $H_m\backslash \mathcal H$, so, the fibers of $\mathcal F_m$ have $H_m$ as fundamental groups. 

The space of leaves $X_m$ of $\mathcal F_m$ is canonically isometric to  the quotient of $\tilde M$ by the action of the group $\mathbb (\ZZ\times H_1)\cdot \mathcal H_m$. Hence, $X_m$  is the  circle of length $\frac 1 m$.   Thus, the leaves of $\mathcal F_m$ are  fibers of a Riemannian submersion 
$P_m:M\to X_m$.  As explained above, its fibers have fundamental groups $H_m$ and constant mean curvature, and this mean curvature does not depend on the fiber.  Taking the fiber with minimal volume, we deduce that  at this fiber the mean curvature is $0$, thus all fibers of $P_m$ are minimal.  Composing $P_m$ with an isoparametric function $g_m:X_m\to \RR$, we obtain an isoparametric function $f_m:=g_m\circ P_m:M\to \RR$ by  \cite[Proposition 2.4]{GT13}.
  The isoparametric foliation induced by $f_m$ is exactly $\mathcal F_m$.  But the fundamental groups of the fibers  of $\mathcal F_m$  are different, for different $m$.  This implies the claim for $\dim (M)=4$ and thus also in higher dimensions by taking products with a torus.
\end{proof}

\section{Topological background}\label{S:top-res}
In this final section, we collect the differential-topological results used in the previous sections. To high-dimensional topologists, essentially all of them are either clear or well-known; we merely explain how to extract them from the literature for more geometrically minded readers. We adopt the notation from \cref{S:conc-iso-diffeo}. In particular, for closed manifolds $M$ we write $\Diff(M)=\Diff(M,M)$ and $\Homeo(M)=\Homeo(M,M)$ for the topological groups of diffeomorphisms or homeomorphisms of $M$. If $M$ is compact but may have boundary, we require diffeomorphisms or homeomorphisms to agree with the identity in a neighbourhood of the boundary, which we indicate by a $\partial$-subscript. If we restrict to orientation-preserving diffeomorphisms or homeomorphisms, we add a $+$-superscript.

Several of the arguments below involve the commutative monoid $\Theta_n$ of closed oriented smooth manifolds homeomorphic $\sphere^n$, up to orientation-preserving diffeomorphism, with the monoid structure incuced by taking connected sums. Writing $\sphere^n=D^n_+\cup D^n_-$ for the decomposition of $\sphere^n$ into hemispheres, the monoid $\Theta_n$ is  related to various diffeomorphism groups by morphisms 
\begin{equation}\label{equ:diagram-section5}
\begin{tikzcd}[row sep=0.35cm, column sep=1cm]
\pi_0(\Diff_\partial(D_+^n))\rar{\mathrm{ext}_n}\arrow["p_n",two heads,d, near start]&\pi_0(\Diff^+(\sphere^n))\arrow["p_n",two heads,d,near start]&\\
\widetilde{\pi}_0(\Diff_\partial(D_+^n))\rar{\widetilde{\mathrm{ext}}_n}&\widetilde{\pi}_0(\Diff^+(\sphere^n))\rar{\mathrm{glue}_n}&\Theta_{n+1}.
\end{tikzcd}
\end{equation}
The maps $p_n$ are epimorphisms that send an isotopy class to its concordance class, the maps $\mathrm{ext}_n$ and $\widetilde{\mathrm{ext}}_n$ are induced by extending diffeomorphisms via $\mathrm{id}_{D^n_-}$, and the right map $\mathrm{glue}_n$ by sending a diffeomorphism $\varphi$ to $\smash{D^{n+1}_+\cup_{\varphi}D^{n+1}_-}$. 
\begin{lemma}\label{lem:diffeospheregroups}Consider the diagram \eqref{equ:diagram-section5}.
\begin{enumerate}
\item\label{enum:theta1} The maps $\mathrm{ext}_n$ and $\widetilde{\mathrm{ext}}_n$ are isomorphisms for all $n$. 
\item\label{enum:theta3}  The map $\mathrm{glue}_n$ is injective for all $n$ and surjective for $n\neq 3$.
\item\label{enum:theta2}  The maps $p_n$ are injective (and thus isomorphisms) for $n\neq 4$. 
\item\label{enum:theta4}  The monoid $\Theta_{n+1}$ is trivial for $n\le 5$ except possibly for $n=3$. For $n\ge6$ it is a finite abelian group.
\end{enumerate}
In particular, for $n\neq 3,4$, all groups in \eqref{equ:diagram-section5} are isomorphic and finite.
\end{lemma}
\begin{proof}
By the parametrised isotopy extension sequence (see e.g.\ \cite[Theorem 6.1.1]{WallDT}), restriction along $D^n_-\subset \sphere^n$ yields a fibration $\Diff^+(\sphere^n)\rightarrow \mathrm{Emb}^+(D^n_-,\sphere^n)$ to the space of orientation-preserving embeddings $D^n_-\hookrightarrow \sphere^n$. The fiber over the inclusion is $\Diff_\partial(D_+^n)$. Taking derivatives induces a homotopy equivalence $\mathrm{Emb}^+(D^n_-,\sphere^n)\simeq \mathrm{GL}^+_{n+1}(\mathbb{R})\simeq \mathrm{SO}(n+1)$, so the standard $\mathrm{SO}(n+1)$-action on $\sphere^n$ yields a section of $\Diff^+(\sphere^n)\rightarrow \mathrm{Emb}^+(D^n_-,\sphere^n)$ up to homotopy. The long exact sequence in homotopy groups then implies that $\mathrm{ext}_n$ is an isomorphism, so $\smash{\widetilde{\mathrm{ext}}_n}$ is surjective too. To show injectivity of $\widetilde{\mathrm{ext}}_n$, we prove that $\widetilde{\mathrm{ext}}_n\circ \mathrm{glue}_n$ is injective: let $[\varphi]$ be a class in the kernel, then there is an orientation-preserving diffeomorphism \[\hspace{-0.5cm}\smash{D^{n+1}_+\cup_{\mathrm{glue}(\varphi)}(\sphere^{n}\times [0,1])\cup_{\sphere^{n}\times \{1\}}D^{n+1}_-\overset{\Phi}{\cong} D^{n+1}_+\cup_{\mathrm{id}}(\sphere^{n}\times [0,1])\cup_{\sphere^{n}\times \{1\}}D^{n+1}_-.}\]
Note that both sides contain $\smash{K\coloneqq {D^{n+1}_+}\cup_{D^n_+\times \{0\}}(D^n_+\times [0,1])\cup_{ D^n_+\times \{1\}}D^{n+1}_-}$ as a submanifold. By combining, firstly,  that the latter is (up to smoothing corners) diffeomorphic to $\DD^{n+1}$ and, secondly, that $\mathrm{Emb}^+(\DD^{n+1},M)$ is connected for any oriented $(n+1)$-manifold $M$ because it is (by taking derivatives) homotopy equivalent to the frame bundle of $M$, we may assume that $\Phi$ agrees with the identity on $K$, so it induces a diffeomorphism between the complements. The latter yields a concordance between $\varphi$ and $\mathrm{id}$, so $\widetilde{\mathrm{ext}}_n\circ \mathrm{glue}_n$ is indeed injective. Together with the surjectivity of $\widetilde{\mathrm{ext}}_n$, this also implies injectivity of $\mathrm{glue}_n$. Surjectivity of $\mathrm{glue}_n$ for $n\neq 3$ is a consequence of Smale's $h$-cobordism theorem \cite[Propositions A and C]{Mil65}. This proves \ref{enum:theta1} and \ref{enum:theta3}. Part \ref{enum:theta2} holds for $n\le 3$ since in this case $\pi_0(\Diff^+(\sphere^n))=0$ \cite{Smale,Cer68}. For $n\ge5$ it holds by Cerf's pseudoisotopy-implies-isotopy theorem \cite{Cer70}. Finally, $\Theta_{n+1}$ is trivial for $n=0,1$ by the classification of $1$- and $2$-manifolds, and for $n=2$ by the Poincaré conjecture. For $n\ge4$, it is a finite abelian group by work of Kervaire--Milnor who also proved that it is trivial for $n=4,5,6$ \cite[Theorem 1.1-1.2, Remark on p.\,505]{KM63}.
\end{proof}

As $\pi_0(\Diff^+(\sphere^n))\le \pi_0(\Diff(\sphere^n))$ has index two, \cref{lem:diffeospheregroups} implies that $\pi_0(\Diff(\sphere^n))$ is finite for $n\neq3,4$. As $\pi_0(\Diff^+(\sphere^n))$ is trivial for $n=3$ by \cite{Cer68}, we get:

\begin{corollary}\label{lem:finitenessdiffeosphere}$\pi_0(\Diff(\sphere^n))$ is finite for all $n\neq 4$.
\end{corollary}

\cref{lem:finitenessdiffeosphere} is not known for $n=4$, but the following suffices for our applications:

\begin{lemma}\label{lem:imagefinitenessdiffeo4sphere}The image of $\smash{\pi_0(\Diff^+(\sphere^4))\xrightarrow{(-)\times \sphere^1} \pi_0(\Diff^+(\sphere^4\times \sphere^1))}$ is trivial.
\end{lemma}
\begin{proof}
Since $\mathrm{ext}_4$ from \eqref{equ:diagram-section5} is an isomorphism by \cref{lem:diffeospheregroups}, it suffices to show that the composition $\pi_0(\Diff_\partial(D_+^4))\rightarrow \pi_0(\Diff_\partial(D_+^4\times \sphere^1))\rightarrow  \pi_0(\Diff^+(\sphere^4\times \sphere^1))$ is trivial. We will see that this is already the case for the first map. Fix a basepoint $*\in D^1_+\subset \sphere^1$ and consider the restriction map $\Diff_\partial(D_+^4\times \sphere^1)\rightarrow \mathrm{Emb}_\partial(D_+^4\times *,D_+^4\times \sphere^1 )$ to the space of embeddings that agree with the inclusion $\mathrm{inc}\colon D_+^4\times *\subset D_+^4\times \sphere^1$ near the boundary. By the parametrised isotopy extension sequence (see e.g.\ \cite[Theorem 6.1.1]{WallDT}), this is a fibration with fiber the space of diffeomorphisms of $D_+^4\times \sphere^1$ that fix $\partial(D_+^4\times \sphere^1)\cup D_+^4\times *$ pointwise. The latter is homeomorphic to $\Diff_\partial(D_+^4\times D^1_-)\cong \Diff_\partial(D^5)$ which is path-connected by \cref{lem:diffeospheregroups}, so the long exact sequence for a fibration implies that the restriction map $\pi_0(\Diff_\partial(D_+^4\times \sphere^1))\rightarrow \pi_0(\mathrm{Emb}_\partial(D_+^4\times *,D_+^4\times \sphere^1 ))$ is injective. It thus suffices to show that the map $\pi_0(\Diff_\partial(D_+^4))\rightarrow  \pi_0(\mathrm{Emb}_\partial(D_+^4\times *,D_+^4\times \sphere^1 ))$ which sends $[\varphi]\in \pi_0(\Diff_\partial(D_+^4))$ to $[\mathrm{inc}\circ\varphi]$ is constant. By identifying $D_+^1\subset \sphere^1$ with $[0,1]$ one sees that this maps sends concordant diffeomorphisms to the same class of embeddings, so it factors over the group $\widetilde{\pi}_0(\Diff_\partial(D_+^4))$. The latter is trivial by \cref{lem:diffeospheregroups}, so the claim follows.
\end{proof}

\begin{lemma}\label{lem:diffeos1sn}
For $n\ge3$, the kernel of the morphism $\pi_0(\Diff(\sphere^1\times \sphere^n))\rightarrow \widetilde{\pi}_0(\Diff(\sphere^1\times \sphere^n))$ induced by sending a diffeomorphism to its concordance class is infinite.
\end{lemma}

\begin{proof}The group $\widetilde{\pi}_0(\Diff(\sphere^1\times \sphere^n))$ is finite for $n\ge3$ (see e.g.\,\cite[Theorem III]{Sat69}), so it suffices to show that $\pi_0(\Diff(\sphere^1\times \sphere^n))$ is infinite. For $n\ge4$, this follows from \cite[Example 1, p.\,408]{HS76}.\footnote{The reference \cite[Example 1, p.\,408]{HS76} is only stated for $\dim(\sphere^1\times \sphere^n)\ge6$, but in the footnote on p.\,401 it is pointed out that their method also goes through in the case $\dim(\sphere^1\times \sphere^n)=5$ using ``a result of K. Igusa''. The latter was later published as \cite{Igusa}.} Finally, the case $n=3$ holds by \cite[Theorem 3.15]{BG21}.
\end{proof}

\begin{remark}
For $n\le 2$, \cref{lem:diffeos1sn} fails: for $n=2$, the group $\pi_0(\Diff(\sphere^1\times \sphere^n))$ is finite by \cite[Theorem 5.1]{Glu62} and for $n\le 1$, it follows from the Dehn-Nielsen theorem that any diffeomorphism concordant to the identity (so in particular homotopic to the identity) is also isotopic to the identity.
\end{remark}

\begin{lemma}\label{L:black-magic}
For a closed smooth manifold $M$ of dimension $n\neq 4$, the forgetful map $\pi_0(\Diff(M))\to \pi_0(\Homeo(M))$ has finite kernel.
\end{lemma}
   
\begin{proof}By the long exact sequence in homotopy groups, it suffices to show that the relative homotopy group $\pi_1(\Homeo(M),\Diff(M);\mathrm{id})$ is finite. The latter is isomorphic to the fundamental group $\pi_1(\mathrm{Sm}(M))$ of the space of smooth structures on $M$ based at the given smooth structure (see \cite[p.\,217]{KS77}). Moreover, by Theorem 2.3 on p.\,235 loc.cit., the space $\mathrm{Sm}(M)$ is homotopy equivalent to the space of lifts of a fixed classifying map $M\rightarrow \mathrm{BHomeo}(\mathbb{R}^d)$ of the topological tangent bundle of $M$ along a model for the forgetful map $\varphi\colon \mathrm{BDiff}(\mathbb{R}^n)\ra \mathrm{BHomeo}(\mathbb{R}^n)$ as a fibration. To show that $\pi_1(\mathrm{Sm}(M))$ is finite, it is by obstruction theory therefore enough to show that the cohomology groups $\oH^i(M;\pi_j(\Homeo(\mathbb{R}^n)/\Diff(\mathbb{R}^n)))$ are finite whenever $j=i$ or $j=i+1$; here $\Homeo(\mathbb{R}^n)/\Diff(\mathbb{R}^n)$ is the fiber of $\varphi$. Since closed smooth $n$-manifolds admit finite $n$-dimensional CW-structures, this follows from showing that the group $\pi_i(\Homeo(\mathbb{R}^n)/\Diff(\mathbb{R}^n))$ is finite for $i\le n+1$. For $n\le 3$, these groups vanish by 5.0 (6) \& (7) on p.\,246 loc.cit.. For $n\ge5$, taking products with $\mathbb{R}$ induces an isomorphism $\pi_i(\Homeo(\mathbb{R}^n)/\Diff(\mathbb{R}^n))\cong \pi_i(\Homeo(\mathbb{R}^{n+1})/\Diff(\mathbb{R}^{n+1}))$  for $i\le n+1$ by 5.0 (3) on p.\,246 loc.cit.. Moreover, in this range this group is by 5.0 (5) on p.\,246 loc.cit.\,isomorphic to the finite group $\Theta_i$ featuring in \cref{lem:diffeospheregroups}.
\end{proof}

 \begin{remark}\label{R:when-fail}
\cref{L:black-magic} fails for $n=4$. For example, it follows from works of Quinn and Ruberman that there are $k,l\ge0$ such that for the connected sum $M=(\mathbb{C}P^2)^{\sharp k}\sharp(\overline{\mathbb{C}P^2})^{\sharp  l}$, the kernel of $\pi_0(\Diff(M))\to \pi_0(\Homeo(M))$ is not even finitely generated: by \cite[Theorem A]{Rub98}, the kernel of the action $\pi_0(\Diff(M))\ra \mathrm{Aut}(H_2(M))$ on second homology is not finitely generated, and by \cite[Theorem 1.1]{Qui86} this kernel agrees with the kernel of $\pi_0(\Diff(M))\ra \pi_0(\Homeo(M))$.
\end{remark}

\begin{lemma}\label{lem:interior-unique}Two compact smooth manifolds $M_0$ and $M_1$ of dimension $n\neq 4,5$ with simply-connected boundaries $\partial M_0$ and $\partial M_1$ are diffeomorphic if and only if their interiors $\mathrm{int}(M_0)$ and $\mathrm{int}(M_1)$ are diffeomorphic.
\end{lemma}

\begin{proof}The ``only if'' direction follows since any diffeomorphism restricts to a diffeomorphism between the interiors by invariance of domain. To prove the ``if'' direction, fix a diffeomorphism $ \mathrm{int}(M_0)\cong \mathrm{int}(M_1)$ and smooth collars $\partial M_i\times [0,2)\subset M_i$ of $\partial M_i\times\{0\}=\partial M_i$ in $M_i$ for $i=0,1$. Viewing $M_0\backslash (\partial M_0\times[0,1])$ as a submanifold of $M_1$ via the embedding \begin{equation}\label{equ:composed-emb}M_0\backslash (\partial M_0\times[0,1])\subset \mathrm{int}(M_0)\cong\mathrm{int}(M_1)\subset M_1,\end{equation} 
the complement $K\coloneqq  M_1\backslash (M_0\backslash (\partial M_0\times[0,1]))$ is a cobordism 
from $\partial M_1$ to $\partial M_0$. It suffices to show that $K$ is an $h$-cobordism since this implies by the $h$-cobordism theorem that there is a diffeomorphism $K\cong \partial M_0\times [0,1]$ extending the identity on $\partial M_0=\partial M_0\times\{1\}$ which in turn gives the claim: \[M_1= K\cup_{\partial M_0 \times 1}(M_0\backslash (\partial M_0\times[0,1))) \cong  \partial_0M\times [0,1]\cup_{\partial M_0 \times 1} (M_0\backslash (\partial M_0\times[0,1)))= M_0.\]
Since the $\partial M_i$ are simply connected, for $K$ to be an h-cobordism, it suffices that $K$ is simply connected and the relative homology group $\oH_*(K,\partial_0M)$ vanishes in all degrees. To see this, note that both inclusions in \eqref{equ:composed-emb} are isotopy equivalences, i.e.\,admit  embeddings in the other direction that are two-sided inverses up to isotopy of embeddings (to see this, push $M_i$ into their interiors using the collar). In particular, the composition \eqref{equ:composed-emb} is a homotopy equivalence. This implies, firstly, that $K$ is simply connected by applying the Seifert--van--Kampen theorem to $M_1=K\cup_{\partial M_0 \times 1}(M_0\backslash (\partial M_0\times[0,1)))$, and secondly that the groups  $\oH_*(M_1, M_0\backslash \partial M_0\times[0,1))$ vanish. The latter are isomorphic to $\oH_*(K, \partial M_0)$ by excision, so the claim follows.\end{proof}

\begin{remark}\cref{lem:interior-unique} fails for $n=5$: consider two non-diffeomorphic homotopy equivalent simply-connected closed smooth $4$-manifolds $P_0$ and $P_1$ with vanishing signature, choose a compact $5$-manifold $M_0$ that bounds $P_0$ (this exists since the signature vanishes), use \cite[Theorem 2]{Wal64} to find an $h$-cobordism $W$ between $P_0$ and $P_1$, and set $M_1\coloneqq  M_0\cup_{P_0}W$. Then $M_0\not\cong M_1$ since $P_0\not\cong P_1$, but $\mathrm{int}(M_0)\cong \mathrm{int}(M_1)$ since $W\backslash P_1\cong P_0\times [0,1)$ by a ``swindle'', using that $W$ is an invertible cobordism; see \cite[Theorem 4]{Sta65}.  
\end{remark}

\bibliographystyle{amsalpha}

\begin{thebibliography}{GWY19}

\bibitem[Abr83]{Abr83}
U.~Abresch, \emph{Isoparametric hypersurfaces with four or six distinct
  principal curvatures. {N}ecessary conditions on the multiplicities}, Math.
  Ann. \textbf{264} (1983), no.~3, 283--302. \MR{714104}

\bibitem[AKP24]{AKP24}
A.~Alexander, V.~Kapovitch, and A.~Petrunin, \emph{Alexandrov
  geometry---foundations}, Graduate Studies in Mathematicsa, vol. 236, American
  Mathematical Society, Providence, RI, [2024] \copyright 2024. \MR{4734965}

\bibitem[Bel03]{Bel03}
I.~Belegradek, \emph{Vector bundles with infinitely many souls}, Proc. Amer.
  Math. Soc. \textbf{131} (2003), no.~7, 2217--2221. \MR{1963770}

\bibitem[BG00]{BG00}
V.~N. Berestovskii and L.~Guijarro, \emph{A metric characterization of
  {R}iemannian submersions}, Ann. Global Anal. Geom. \textbf{18} (2000), no.~6,
  577--588. \MR{1800594}

\bibitem[BG21]{BG21}
R.~Budney and D.~Gabai, \emph{Knotted 3-balls in {$S^4$}}, arXiv:1912.09029
  (2021).

\bibitem[Car38]{Car38}
\'E. Cartan, \emph{Familles de surfaces isoparam\'etriques dans les espaces
  \`a{} courbure constante}, Ann. Mat. Pura Appl. \textbf{17} (1938), no.~1,
  177--191. \MR{1553310}

\bibitem[Car39]{Car39}
E.~Cartan, \emph{Sur des familles remarquables d'hypersurfaces
  isoparam\'etriques dans les espaces sph\'eriques}, Math. Z. \textbf{45}
  (1939), 335--367. \MR{169}

\bibitem[Cer68]{Cer68}
J.~Cerf, \emph{Sur les diff\'eomorphismes de la sph\`ere de dimension trois
  {$(\Gamma \sb{4}=0)$}}, Lecture Notes in Mathematics, vol. No. 53,
  Springer-Verlag, Berlin-New York, 1968. \MR{229250}

\bibitem[Cer69]{Cern69}
A.~V. Cernavskii, \emph{Local contractibility of the group of homeomorphisms of
  a manifold}, Mat. Sb. (N.S.) \textbf{79(121)} (1969), 307--356. \MR{259925}

\bibitem[Cer70]{Cer70}
J.~Cerf, \emph{La stratification naturelle des espaces de fonctions
  diff\'erentiables r\'eelles et le th\'eor\`eme de la pseudo-isotopie}, Inst.
  Hautes \'Etudes Sci. Publ. Math. (1970), no.~39, 5--173. \MR{292089}

\bibitem[Che70]{Cheeger}
Jeff Cheeger, \emph{Finiteness theorems for {R}iemannian manifolds}, Amer. J.
  Math. \textbf{92} (1970), 61--74. \MR{263092}

\bibitem[Chi20]{Chi}
Q.-S. Chi, \emph{Isoparametric hypersurfaces with four principal curvatures,
  {IV}}, J. Differential Geom. \textbf{115} (2020), no.~2, 225--301.
  \MR{4100704}

\bibitem[DeV23]{De23}
J.~DeVito, \emph{Counterexamples to the nonsimply connected double soul
  conjecture}, Pacific J. Math. \textbf{325} (2023), no.~2, 239--254.
  \MR{4662642}

\bibitem[DN85]{DN85}
J.~Dorfmeister and E.~Neher, \emph{Isoparametric hypersurfaces, case
  {$g=6,\;m=1$}}, Comm. Algebra \textbf{13} (1985), no.~11, 2299--2368.
  \MR{807479}

\bibitem[DV16]{DV16}
M.~Dom\'inguez-V\'azquez, \emph{Isoparametric foliations on complex projective
  spaces}, Trans. Amer. Math. Soc. \textbf{368} (2016), no.~2, 1211--1249.
  \MR{3430362}

\bibitem[DVG18]{DG18}
M.~Dom\'inguez-V\'azquez and C.~Gorodski, \emph{Polar foliations on
  quaternionic projective spaces}, Tohoku Math. J. (2) \textbf{70} (2018),
  no.~3, 353--375. \MR{3856771}

\bibitem[EE67]{EarleEells}
  C.~J.~Earle Jr. and J.~Eells, Foliations and fibrations, J. Differential Geometry {\bf 1} (1967), no.~1, 33--41.

\bibitem[FKM81]{FKM}
D.~Ferus, H.~Karcher, and H.~F. M\"unzner, \emph{Cliffordalgebren und neue
  isoparametrische {H}yperfl\"achen}, Math. Z. \textbf{177} (1981), no.~4,
  479--502. \MR{624227}

\bibitem[Ge16]{Ge16}
J.~Ge, \emph{Isoparametric foliations, diffeomorphism groups and exotic smooth
  structures}, Adv. Math. \textbf{302} (2016), 851--868. \MR{3545943}

\bibitem[GH87]{GH87}
K.~Grove and S.~Halperin, \emph{Dupin hypersurfaces, group actions and the
  double mapping cylinder}, J. Differential Geom. \textbf{26} (1987), no.~3,
  429--459. \MR{910016}

\bibitem[Glu62]{Glu62}
H.~Gluck, \emph{The embedding of two-spheres in the four-sphere}, Trans. Amer.
  Math. Soc. \textbf{104} (1962), 308--333. \MR{146807}

\bibitem[GR15]{GR15}
J.~Ge and M.~Radeschi, \emph{Differentiable classification of 4-manifolds with
  singular {R}iemannian foliations}, Math. Ann. \textbf{363} (2015), no.~1-2,
  525--548. \MR{3394388}

\bibitem[GT13]{GT13}
J.~Ge and Z.~Tang, \emph{Isoparametric functions and exotic spheres}, J. Reine
  Angew. Math. \textbf{683} (2013), 161--180. \MR{3181552}

\bibitem[Gui98]{Gui98}
L.~Guijarro, \emph{Improving the metric in an open manifold with nonnegative
  curvature}, Proc. Amer. Math. Soc. \textbf{126} (1998), no.~5, 1541--1545.
  \MR{1443388}

\bibitem[GWY19]{GWY}
K.~Grove, B.~Wilking, and J.~Yeager, \emph{Almost non-negative curvature and
  rational ellipticity in cohomogeneity two}, Ann. Inst. Fourier (Grenoble)
  \textbf{69} (2019), no.~7, 2921--2939, With an appendix by Steve Halperin.
  \MR{4286825}

\bibitem[Har16]{Harvey}
J.~Harvey, \emph{Equivariant {A}lexandrov geometry and orbifold finiteness}, J.
  Geom. Anal. \textbf{26} (2016), no.~3, 1925--1945. \MR{3511464}

\bibitem[Hir94]{Hir94}
M.~W. Hirsch, \emph{Differential topology}, Graduate Texts in Mathematics,
  vol.~33, Springer-Verlag, New York, 1994, Corrected reprint of the 1976
  original. \MR{1336822}

\bibitem[HS76]{HS76}
W.~C. Hsiang and R.~W. Sharpe, \emph{Parametrized surgery and isotopy}, Pacific
  J. Math. \textbf{67} (1976), no.~2, 401--459. \MR{494165}

  
\bibitem[Igu82]{Igusa}  
  K.~Igusa, What happens to Hatcher and Wagoner's formulas for $\pi \sb{0}C(M)$\ when the first Postnikov invariant of $M$\ is nontrivial?, in {\it Algebraic $K$-theory, number theory, geometry and analysis (Bielefeld, 1982)}, 104--172, Lecture Notes in Math., 1046, Springer, Berlin.

\bibitem[Imm08]{Imm08}
S.~Immervoll, \emph{On the classification of isoparametric hypersurfaces with
  four distinct principal curvatures in spheres}, Ann. of Math. (2)
  \textbf{168} (2008), no.~3, 1011--1024. \MR{2456889}

\bibitem[Kap07]{Kap}
V.~Kapovitch, \emph{Perelman's stability theorem}, Surveys in differential
  geometry. {V}ol. {XI}, Surv. Differ. Geom., vol.~11, Int. Press, Somerville,
  MA, 2007, pp.~103--136. \MR{2408265}

\bibitem[KL21]{KL21}
V.~Kapovitch and A.~Lytchak, \emph{Remarks on manifolds with two-sided
  curvature bounds}, Anal. Geom. Metr. Spaces \textbf{9} (2021), no.~1, 53--64.
  \MR{4277411}

\bibitem[KL22]{KL22}
\bysame, \emph{The structure of submetries}, Geom. Topol. \textbf{26} (2022),
  no.~6, 2649--2711. \MR{4521250}

\bibitem[KM63]{KM63}
M.~A. Kervaire and J.~W. Milnor, \emph{Groups of homotopy spheres. {I}}, Ann.
  of Math. (2) \textbf{77} (1963), 504--537. \MR{148075}

\bibitem[Kol02]{Kol02}
A.~Kollross, \emph{A classification of hyperpolar and cohomogeneity one
  actions}, Trans. Amer. Math. Soc. \textbf{354} (2002), no.~2, 571--612.
  \MR{1862559}

\bibitem[KS77]{KS77}
R.~C. Kirby and L.~C. Siebenmann, \emph{Foundational essays on topological
  manifolds, smoothings, and triangulations}, Annals of Mathematics Studies,
  vol. No. 88, Princeton University Press, Princeton, NJ; University of Tokyo
  Press, Tokyo, 1977, With notes by John Milnor and Michael Atiyah. \MR{645390}

\bibitem[LC37]{LC37}
T.~Levi-Civita, \emph{Famiglie di superficie isoparametriche nell'ordinario
  spazio euclideo}, Atti Accad. Naz. Lincei, Rend., VI. Ser. \textbf{26}
  (1937), 355--362 (Italian).


\bibitem[LY24]{LiYang}
M.~{Li} and L.~{Yang}, \emph{{The Equivalence of the Existences of Transnormal
  and Isoparametric Functions on Compact Manifolds}}, arXiv:2410.21016 (2024).

\bibitem[Lyt05]{Lyt05}
A.~Lytchak, \emph{Almost convex subsets}, Geom. Dedicata \textbf{115} (2005),
  201--218. \MR{2180048}

\bibitem[Lyt09]{Lyt09}
\bysame, \emph{Notes on the {J}acobi equation}, Differential Geom. Appl.
  \textbf{27} (2009), no.~2, 329--334. \MR{2503983}

\bibitem[Lyt14]{Lyt14}
\bysame, \emph{Polar foliations of symmetric spaces}, Geom. Funct. Anal.
  \textbf{24} (2014), no.~4, 1298--1315. \MR{3248486}

\bibitem[Mil65]{Mil65}
J.~Milnor, \emph{Lectures on the {$h$}-cobordism theorem}, Princeton University
  Press, Princeton, NJ, 1965, Notes by L. Siebenmann and J. Sondow. \MR{190942}

\bibitem[Miy13]{Miy}
R.~Miyaoka, \emph{Isoparametric hypersurfaces with {$(g,m)=(6,2)$}}, Ann. of
  Math. (2) \textbf{177} (2013), no.~1, 53--110. \MR{2999038}

\bibitem[MR20]{MR20}
R.~A.~E. Mendes and M.~Radeschi, \emph{Laplacian algebras, manifold submetries
  and the inverse invariant theory problem}, Geom. Funct. Anal. \textbf{30}
  (2020), no.~2, 536--573. \MR{4108615}

\bibitem[{Mü}80]{Mun80}
H.~F. {Münzner}, \emph{Isoparametrische {H}yperfl\"achen in {S}ph\"aren},
  Math. Ann. \textbf{251} (1980), no.~1, 57--71. \MR{583825}

\bibitem[{Mü}81]{Mun81}
\bysame, \emph{Isoparametrische {H}yperfl\"achen in {S}ph\"aren. {II}. \"uber
  die {Z}erlegung der {S}ph\"are in {B}allb\"undel}, Math. Ann. \textbf{256}
  (1981), no.~2, 215--232. \MR{620709}

\bibitem[Pet84]{Peters2}
S.~Peters, \emph{Cheeger's finiteness theorem for diffeomorphism classes of
  {R}iemannian manifolds}, J. Reine Angew. Math. \textbf{349} (1984), 77--82.
  \MR{743966}

\bibitem[Pet87]{Peters}
\bysame, \emph{Convergence of {R}iemannian manifolds}, Compositio Math.
  \textbf{62} (1987), no.~1, 3--16. \MR{892147}

\bibitem[Pug87]{Pu}
C.~C. Pugh, \emph{The {$C^{1,1}$} conclusions in {G}romov's theory}, Ergodic
  Theory Dynam. Systems \textbf{7} (1987), no.~1, 133--147. \MR{886375}

\bibitem[QT15]{QT15}
C.~Qian and Z.~Tang, \emph{Isoparametric functions on exotic spheres}, Adv.
  Math. \textbf{272} (2015), 611--629. \MR{3303243}

\bibitem[Qui86]{Qui86}
F.~Quinn, \emph{Isotopy of {$4$}-manifolds}, J. Differential Geom. \textbf{24}
  (1986), no.~3, 343--372. \MR{868975}

\bibitem[Rub99]{Rub98}
D.~Ruberman, \emph{A polynomial invariant of diffeomorphisms of 4-manifolds},
  Proceedings of the {K}irbyfest ({B}erkeley, {CA}, 1998), Geom. Topol.
  Monogr., vol.~2, Geom. Topol. Publ., Coventry, 1999, pp.~473--488.
  \MR{1734421}

\bibitem[Sat69]{Sat69}
H.~Sato, \emph{Diffeomorphism group of {$S\sp{p}\times S\sp{q}$} and exotic
  spheres}, Quart. J. Math. Oxford Ser. (2) \textbf{20} (1969), 255--276.
  \MR{253369}

\bibitem[Sma59]{Smale}
S.~J. Smale, Diffeomorphisms of the $2$-sphere, Proc. Amer. Math. Soc. {\bf 10} (1959), 621--626.

\bibitem[Sif16]{Sif16}
A.~Siffert, \emph{Classification of isoparametric hypersurfaces in spheres with
  {$(g,m)=(6,1)$}}, Proc. Amer. Math. Soc. \textbf{144} (2016), no.~5,
  2217--2230. \MR{3460180}

\bibitem[Sta65]{Sta65}
J.~R. Stallings, \emph{On infinite processes leading to differentiability in
  the complement of a point}, Differential and {C}ombinatorial {T}opology ({A}
  {S}ymposium in {H}onor of {M}arston {M}orse), Princeton Univ. Press,
  Princeton, NJ, 1965, pp.~245--254. \MR{180983}

\bibitem[Sto99]{Sto}
S.~Stolz, \emph{Multiplicities of {D}upin hypersurfaces}, Invent. Math.
  \textbf{138} (1999), no.~2, 253--279. \MR{1720184}

\bibitem[Tap00]{Tapp}
K.~Tapp, \emph{Bounded {R}iemannian submersions}, Indiana Univ. Math. J.
  \textbf{49} (2000), no.~2, 637--654. \MR{1793685}

\bibitem[Tho91]{Tho91}
G.~Thorbergsson, \emph{Isoparametric foliations and their buildings}, Ann. of
  Math. (2) \textbf{133} (1991), no.~2, 429--446. \MR{1097244}

\bibitem[TT95]{TT95}
C.-L. Terng and G.~Thorbergsson, \emph{Submanifold geometry in symmetric
  spaces}, J. Differential Geom. \textbf{42} (1995), no.~3, 665--718.
  \MR{1367405}

\bibitem[Wal64]{Wal64}
C.~T.~C. Wall, \emph{On simply-connected {$4$}-manifolds}, J. London Math. Soc.
  \textbf{39} (1964), 141--149. \MR{163324}

\bibitem[Wal16]{WallDT}
\bysame, \emph{Differential topology}, Cambridge Studies in Advanced
  Mathematics, vol. 156, Cambridge University Press, Cambridge, 2016.
  \MR{3558600}

\bibitem[Wan87]{Wan}
Q.~M. Wang, \emph{Isoparametric functions on {R}iemannian manifolds. {I}},
  Math. Ann. \textbf{277} (1987), no.~4, 639--646. \MR{901710}

\end{thebibliography}
\providecommand{\bysame}{\leavevmode\hbox to3em{\hrulefill}\thinspace}
\providecommand{\MR}{\relax\ifhmode\unskip\space\fi MR }
\providecommand{\MRhref}[2]{%
  \href{http://www.ams.org/mathscinet-getitem?mr=#1}{#2}
}
\providecommand{\href}[2]{#2}

\end{document}